\documentclass[a4paper,12pt]{article}
\usepackage{a4wide}
\usepackage{amsmath}
\usepackage{amssymb}
\usepackage{amsthm}
\usepackage{latexsym}
\usepackage{graphicx}
\usepackage[english]{babel}
\usepackage{makeidx}

\newtheorem{prop}[subsection]{Proposition}
\newtheorem{conj}[subsection]{Conjecture}
\newtheorem{teor}[subsection]{Theorem}
\newtheorem{lema}[subsection]{Lemma}
\newtheorem{cor} [subsection]{Corollary}
\newcommand{\Zng}{$\mathbb Z^n$-graded $S$-module}

\def\sdepth{\operatorname{sdepth}}
\def\depth{\operatorname{depth}}

\begin{document}
\selectlanguage{english}
\frenchspacing

\begin{center}
\textbf{On the Stanley depth of the path ideal of a cycle graph}
\vspace{10pt}

\large
Mircea Cimpoea\c s
\end{center}
\normalsize

\begin{abstract}
We give tight bounds for the Stanley depth of the quotient ring of the path ideal of a cycle graph. In particular,
we prove that it satisfies the Stanley inequality.


\end{abstract}

\section*{Introduction}

Let $K$ be a field and $S=K[x_1,\ldots,x_n]$ the polynomial ring over $K$.
Let $M$ be a \Zng. A \emph{Stanley decomposition} of $M$ is a direct sum $\mathcal D: M = \bigoplus_{i=1}^rm_i K[Z_i]$ as a $\mathbb Z^n$-graded $K$-vector space, where $m_i\in M$ is homogeneous with respect to $\mathbb Z^n$-grading, $Z_i\subset\{x_1,\ldots,x_n\}$ such that $m_i K[Z_i] = \{um_i:\; u\in K[Z_i] \}\subset M$ is a free $K[Z_i]$-submodule of $M$. We define $\sdepth(\mathcal D)=\min_{i=1,\ldots,r} |Z_i|$ and $\sdepth(M)=\max\{\sdepth(\mathcal D)|\;\mathcal D$ is a Stanley decomposition of $M\}$. The number $\sdepth(M)$ is called the \emph{Stanley depth} of $M$. 

Herzog, Vladoiu and Zheng show in \cite{hvz} that $\sdepth(M)$ can be computed in a finite number of steps if $M=I/J$, where $J\subset I\subset S$ are monomial ideals. In \cite{rin}, Rinaldo give a computer implementation for this algorithm, in the computer algebra system $\mathtt{CoCoA}$ \cite{cocoa}. In \cite{apel}, J.\ Apel restated a conjecture firstly given by Stanley in \cite{stan}, namely that $\sdepth(M)\geq\depth(M)$ for any \Zng $\;M$. This conjecture proves to be false, in general, for $M=S/I$ and $M=J/I$, where $0\neq I\subset J\subset S$ are monomial ideals, see \cite{duval}. For a friendly introduction in the thematic of Stanley depth, we refer the reader \cite{her}.

Let $\Delta \subset 2^{[n]}$ be a simplicial complex. A face $F\in\Delta$ is called a \emph{facet}, if $F$ is maximal with respect to inclusion. We denote $\mathcal F(\Delta)$ the set of facets of $\Delta$. If $F\in\mathcal F(\Delta)$, we denote $x_F=\prod_{j\in F}x_j$. Then the \emph{facet ideal $I(\Delta)$} associated to $\Delta$ is the squarefree monomial ideal $I=(x_F\;:\;F\in \mathcal F(\Delta))$ of $S$. The facet ideal was studied by Faridi \cite{faridi} from the \texttt{depth} perspective.

The \emph{line graph} of lenght $n$, denoted by $L_n$, is a graph with the vertex set $V=[n]$ and the edge set $E=\{\{1,2\},\{2,3\},\ldots,\{n-1,n\}\}$. Let $\Delta_{n,m}$ be the simplicial complex with the set of facets 
$\mathcal F(\Delta_{n,m})=\{\{1,2,\ldots,m\}, \{2,3,\ldots,m+1\},\ldots, \{n-m+1,n-m+2,\ldots,n\}\}$, where $1\leq m\leq n$.
We denote $I_{n,m} =(x_1x_2\cdots x_m,  x_2x_3\cdots x_{m+1}, \ldots, x_{n-m+1}x_{n-m+2}\cdots x_n )$
, the associated facet ideal. Note that $I_{n,m}$ is the $m$-path ideal of the graph $L_n$, provided with the direction given by $1<2<\ldots <n$, see \cite{tuy} for further details.

\footnotetext[1]{The support from grant ID-PCE-2011-1023 of Romanian Ministry of Education, Research and Innovation is gratefully acknowledged.}

According to \cite[Theorem 1.2]{tuy}, 
$$pd(S/I_{n,m}) = \begin{cases} \frac{2(n-d)}{m+1},\; n\equiv d (mod\;(m+1))\;with\; 0 \leq d\leq m-1, \\ 
\frac{2n-m+1}{m+1},\; n\equiv m (mod\;(m+1)).
 \end{cases}$$ 
By Auslander-Buchsbaum formula (see \cite{real}), it follows that $\depth(S/I_{n,m})=n-pd(S/I_{n,m})$ and, by a straightforward computation, we can see $\depth(S/I_{n,m}) =  n+1 - \left\lfloor \frac{n+1}{m+1} \right\rfloor - \left\lceil \frac{n+1}{m+1} \right\rceil=:\varphi(n,m)$. 
We proved in \cite{path} that $\sdepth(S/I_{n,m})=\varphi(n,m)$.

The \emph{cycle graph} of lenght $n$, denoted by $C_n$, is a graph with the vertex set $V=[n]$ and the edge set $E=\{\{1,2\},\{2,3\},\ldots,\{n-1,n\},\{n,1\}\}$. Let $\bar \Delta_{n,m}$ be the simplicial complex with the set of facets 
$\mathcal F(\bar \Delta_{n,m})=\{\{1,2,\ldots,m\}, \{2,3,\ldots,m+1\},\cdots,\linebreak \{n-m+1,n-m+2,\ldots,n\}, 
\{n-m+2,\ldots,n,1\}, \ldots, \{n,1,\ldots,m-1\}\}$. We denote $J_{n,m} =(x_1x_2\cdots x_m,  x_2x_3\cdots x_{m+1}, \ldots, x_{n-m+1}x_{n-m+2}\cdots x_n, \ldots, x_nx_1\cdots x_{m-1} )$, the associated facet ideal. Note that $J_{n,m}$ is the $m$-path ideal of the graph $C_n$.

Let $p=\left\lfloor \frac{n}{m+1} \right\rfloor$ and $d=n-(m+1)p$.
According to \cite[Corollary 5.5]{far}, 
$$pd(S/J_{n,m}) = \begin{cases} 2p+1,\; d\neq 0, \\ 
                                2p,\; d=0.
\end{cases}$$ 
By Auslander-Buchsbaum formula, it follows that $\depth(S/J_{n,m})=n-pd(S/J_{n,m})=n-\left\lfloor \frac{n}{m+1} \right\rfloor - \left\lceil \frac{n}{m+1} \right\rceil=:\psi(n,m)$. Note that $\psi(n,m)=\varphi(n-1,m)$.
Our main result is Theorem $1.4$, in which we prove that $\varphi(n,m)\geq \sdepth(S/J_{n,m})\geq \psi(n,m)$.
We also prove that, $\sdepth(J_{n,m}/I_{n,m}) = \depth(J_{n,m}/I_{n,m}) = \psi(n,m)+m-1$, see Proposition $1.6$. These results generalize 
\cite[Theorem 1.9]{ciclu} and \cite[Proposition 1.10]{ciclu}.

\section{Main results}

First, we recall the well known Depth Lemma, see for instance \cite[Lemma 1.3.9]{real}. 

\begin{lema}(Depth Lemma)
If $0 \rightarrow U \rightarrow M \rightarrow N \rightarrow 0$ is a short exact sequence of modules over a local ring $S$, or a Noetherian graded ring with $S_0$ local, then

a) $\depth M \geq \min\{\depth N,\depth U\}$.

b) $\depth U \geq \min\{\depth M,\depth N +1 \}$.

c) $\depth N\geq \min\{\depth U - 1,\depth M\}$.
\end{lema}

In \cite{asia}, Asia Rauf proved the analog of Lemma $1.1(a)$ for $\sdepth$:

\begin{lema}
Let $0 \rightarrow U \rightarrow M \rightarrow N \rightarrow 0$ be a short exact sequence of $\mathbb Z^n$-graded $S$-modules. Then:
$ \sdepth(M) \geq \min\{\sdepth(U),\sdepth(N) \}$. 
\end{lema}

The following result is well known. However, we present an original proof.

\begin{lema}
Let $I\subset S$ be a nonzero proper monomial ideal. Then, $I$ is principal if and only if $\sdepth(S/I)=n-1$.
\end{lema}

\begin{proof}
Assume $\sdepth(S/I)=n-1$ and let $S/I=\bigoplus_{i=1}^r u_iK[Z_i]$ be a Stanley decomposition with $|Z_i|=n-1$ for all $i$, and $u_i\in S$ monomials. Since $1\notin I$, we may assume that $u_1=1$. Let $x_{j_1}$ be the variable which is not in $Z_1$. If $x_{j_1}\in I$, since $S/(x_{j_1})=K[Z_1]$ and $K[Z_1]\subset S/I$, then $I=(x_{j_1})$. Otherwise, we may assume that $u_2=x_{j_1}$.

Let $x_{j_2}$ be the variable which is not in $Z_2$. If $x_{j_1}x_{j_2}\in I$, then, one can easily see that $I=(x_{j_1}x_{j_2})$. If $x_{j_1}x_{j_2}\notin I$, then we may assume $u_3=x_{j_1}x_{j_2}$ and so on. Thus, we have $u_i=x_{j_1}\cdots x_{j_{i-1}}$, for all $1\leq i\leq r+1$, where $x_{j_i}$ is the variable which is not in $Z_i$. Moreover, $I=(u_{r+1})$, and therefore $I$ is principal.

In order to prove the other implication, assume that $I=(u)$ and write $u=\prod_{i=1}^r x_{j_i}$. We let $u_i=\prod_{k=1}^{i-1} x_{j_k}$ and 
$Z_i = \{x_1,\ldots,x_n\} \setminus \{x_{j_i}\}$, for all $1\leq i\leq r$. Then, $S/I=\bigoplus_{i=1}^r u_iK[Z_i]$ is a Stanley decomposition with $|Z_i|=n-1$ for all $i$. Therefore $\sdepth(S/I)=n-1$.
\end{proof}

Our main result, is the following Theorem.

\begin{teor}
$\varphi(n,m)\geq \sdepth(S/J_{n,m})\geq \depth(S/J_{n,m})=\psi(n,m)$.
\end{teor}

\begin{proof}
If $n=m$, then $J_{n,n}=(x_1\ldots x_n)$ is a principal ideal, and, according to Lemma $1.3$ we are done.
Also, if $m=1$, then $J_{n,1}=(x_1,\ldots,x_n)$ and so there is nothing to prove, since $S/J_{n,1}=K$.
The case $m=2$ follows from \cite[Proposition 1.8]{ciclu} and \cite[Theorem 1.9]{ciclu}.

Assume $n>m\geq 3$. If $n=m+1$, then we consider the short exact sequence
$ 0 \rightarrow S/(J_{n,n-1}:x_n) \rightarrow S/J_{n,n-1} \rightarrow S/(J_{n,n-1},x_n) \rightarrow 0$. Note that 
$(J_{n,n-1}:x_n)=(x_1\cdots x_{n-2}, x_2\cdots x_{n-1}, x_3\cdots x_{n-1}x_1,\cdots, x_{n-1}x_1\cdots x_{n-3}) \cong J_{n-1,n-2}S$.
Therefore, by induction hypothesis and \cite[Lemma 3.6]{hvz}, $$\sdepth(S/(J_{n,n-1}:x_n))=\depth(S/(J_{n,n-1}:x_n))=1+\psi(n-1,n-2)=n-2.$$
Also, $(J_{n,n-1},x_n) = (x_1\cdots x_{n-1},x_n)$ and thus $S/(J_{n,n-1},x_n)\cong K[x_1,\ldots,x_{n-1}]/(x_1\cdots x_{n-1})$. Therefore, by Lemma $1.3$, we have $\sdepth(S/(J_{n,n-1},x_n))=n-2 =\depth(S/(J_{n,n-1},x_n))$.

Now, assume $n>m+1>3$. We consider the ideals $L_0=J_{n,m}$, $L_{k+1}=(L_k:x_{n-k})$ and $U_k=(L_k,x_{n-k})$, for $0\leq k\leq m-2$. 
Note that $L_{m-1}=(J_{n,m}:x_{n-m+2}\cdots x_n)=(x_1, x_2\cdots x_{m+1}, \ldots, x_{n-2m+1}\cdots x_{n-m},x_{n-m+1})$. 

If $n-2m\leq 2$, then $L_{m-1}=(x_1,x_{n-m+1})$ and thus $\sdepth(S/L_{m-1})=\depth(S/L_{m-1})=n-2=\varphi(n,m)$, since $\left\lfloor \frac{n+1}{m+1} \right\rfloor=1$ and $\left\lceil \frac{n+1}{m+1} \right\rceil=2$.
If $n-2m>2$, then $S/L_{m-1}\cong K[x_2,\ldots,x_{n-m},x_{n-m+2},\ldots,x_n]/ (x_2\cdots x_{m+1}, \ldots, x_{n-2m+1}\cdots x_{n-m})$ and therefore, by \cite[Lemma 3.6]{hvz} and \cite[Theorem 1.3]{path}, we have 
$\sdepth(S/L_{m-1})=\depth(S/L_{m-1})=n-1-\left\lfloor \frac{n-m}{m+1} \right\rfloor -\left\lceil \frac{n-m}{m+1} \right\rceil = \varphi(n,m)$.
On the other hand, for example by \cite[Proposition 2.7]{mirci}, $\sdepth(S/L_{m-1})\geq \sdepth(S/J_{n,m})$. Thus, $\sdepth(S/J_{n,m})\leq \varphi(n,m)$.

For any $0<k<m$, we have $L_k=(x_1\cdots x_{m-k}, x_2\cdots x_{m+1}, \ldots, x_{n-m-k}\cdots x_{n-k-1}, \linebreak x_{n-m+1}\cdots x_{n-k}, x_{n-m+2}\cdots x_{n-k}x_1,\ldots, x_{n-k}x_1\cdots x_{m-k-1})$. Therefore, $U_k=(x_1\cdots x_{m-k}, \linebreak x_2\cdots x_{m+1},\ldots,x_{n-m-k}\cdots x_{n-k-1}, x_{n-k})$, for $k\leq m-2$. We consider two cases:

$(i)$ If $n-m-k<2$ and $0\leq k\leq m-2$, then $U_k=(x_1\cdots x_{m-k},x_{n-k})$ and therefore $\sdepth(S/U_k)=\depth(S/U_k)=n-2 = \varphi(n,m)$, since $\left\lfloor \frac{n+1}{m+1} \right\rfloor=1$ and $\left\lceil \frac{n+1}{m+1} \right\rceil=2$.

$(ii)$ If $n-m-k\geq 2$, then, for any $0\leq j \leq k\leq m-2$, we consider the ideals $V_{k,j}:=(x_1\cdots x_{m-j}, x_2\cdots x_{m+1},\ldots,x_{n-m-k}\cdots x_{n-k-1})$ in $S_k:=K[x_1,\ldots,n_{n-k-1}]$. Note that $S/U_{k}\cong (S_k/V_{k,k})[x_{n-k+1},\ldots,x_n]$ and thus, by \cite[Lemma 3.6]{hvz}, $\depth(S/U_k) = \depth(S_k/V_{k,k})+k$ and $\sdepth(S/U_k) = \sdepth(S_k/V_{k,k})+k$.

For any $0\leq j <k\leq m-2$, we claim that $V_{k,j}/V_{k,j+1}$ is isomorphic to
$$(K[x_{m-j+2},\ldots,x_{n-k-1}]/(x_{m-j+2}\cdots x_{2m-j+1},\ldots, x_{n-m-k}\cdots x_{n-k-1}))[x_1,\ldots,x_{m-j}].$$
Indeed, if $u\in V_{k,j}\setminus V_{k,j+1}$ is a monomial, then $x_1 \cdots x_{m-j}|u$ and $x_{m-j+1}\nmid u$. Also, $x_{m-j+2}\cdots x_{2m-j+1}\nmid u$, \ldots, $x_{n-m-k}\cdots x_{n-k-1}\nmid u$. Denoting $v=u/(x_1 \cdots x_{m-j})$, we can write $v=v'v''$, with $v'\in K[x_{m-j+2},\ldots,x_{n-k-1}]\setminus (x_{m-j+2}\cdots x_{2m-j+1},\ldots, x_{n-m-k}\cdots x_{n-k-1})$ and $v''\in K[x_1,\ldots,x_{m-j}]$.

By \cite[Lemma 3.6]{hvz} and \cite[Theorem 1.3]{path}, $\sdepth(V_{k,j}/V_{k,j+1})=\depth(V_{k,j}/V_{k,j+1})= m-j+\varphi(n-k-m+j-2,m) = n-k-1 - \left\lfloor \frac{n-m-1-k+j}{m+1} \right\rfloor -\left\lceil \frac{n-m-1-k+j}{m+1} \right\rceil = n-k+1 - \left\lfloor \frac{n-k+j}{m+1} \right\rfloor -\left\lceil \frac{n-k+j}{m+1}\right\rceil \geq \varphi(n,m)-k$. On the other hand, $V_{k,0}=I_{n-k-1,m}$ for any $0\leq k\leq m-2$ and therefore, by \cite[Theorem 1.3]{path}, $\sdepth(S/V_{k,0}) = \depth(S/V_{k,0}) = \varphi(n-k-1,m) = n-k- \left\lfloor \frac{n-k}{m+1} \right\rfloor -\left\lceil \frac{n-k}{m+1} \right\rceil \geq \varphi(n,m)-k$, for any $k\geq 1$. From the short exact sequences $0 \rightarrow V_{k,j}/V_{k,j+1} \rightarrow S/V_{k,j+1} \rightarrow S/V_{k,j} \rightarrow 0$, $0\leq j<k$, Lemma $1.1$ and Lemma $1.2$, it follows that $\sdepth(S/V_{k,j+1})\geq \depth(S/V_{k,j+1})=\varphi(n,m)-k$, for all $0\leq j<k\leq m-2$. Thus $\sdepth(S/U_k) \geq \depth(S/U_k) \geq \varphi(n,m)$, for all $0<k\leq m-2$.On the other hand, $\sdepth(S/V_{0,0}) = \depth(S/V_{0,0}) = \varphi(n-1,m) = \psi(n,m)-m$, and thus $\sdepth(S/U_0)=\depth(S/U_0)=\psi(n,m)$.

Now, we consider short exact sequences \[0 \rightarrow S/L_{k+1} \rightarrow S/L_k \rightarrow S/U_k \rightarrow 0.\;\;for\;0\leq k<m.\]
By Lemma $1.1$ and Lemma $1.2$ we get  $\sdepth(S/L_k) \geq \depth(S/L_k) = \varphi(n,m)$, for any $0<k\leq m-2$, and $\sdepth(S/L_0) \geq \depth(S/L_0) = \psi(n,m)$.
\end{proof}

\begin{cor}
If $\left\lfloor \frac{n+1}{m+1} \right\rfloor = \left\lfloor \frac{n}{m+1} \right\rfloor$  and $\left\lceil \frac{n+1}{m+1} \right\rceil = \left\lceil \frac{n}{m+1} \right\rceil$, then $$\sdepth(S/J_{n,m})=\depth(S/J_{n,m})=\varphi(n,m).$$
\end{cor}

\begin{prop}
$\sdepth(J_{n,m}/I_{n,m}) \geq \depth(J_{n,m}/I_{n,m}) = \psi(n,m)+m-1$.
\end{prop}

\begin{proof}
We claim that $J_{n,m}/I_{n,m}$ is isomorphic to 
\footnotesize
$$x_{n-m+2}\cdots x_nx_1 (K[x_2,\ldots,x_{n-m}]/(x_2\cdots x_{m},x_3\cdots x_{m+2}\ldots,x_{n-2m+1}\cdots x_{n-m}))[x_{n-m+2},\ldots ,x_n,x_1] \oplus $$
\scriptsize
$$  \oplus x_{n-m+3}\cdots x_nx_1x_2 (K[x_3,\ldots,x_{n-m+1}]/(x_3\cdots x_{m},x_4\cdots x_{m+3},\ldots,x_{n-2m+2}\cdots x_{n-m+1}))[x_{n-m+3},\ldots ,x_n,x_1,x_2]  \oplus   $$
\small
$$\cdots \oplus x_nx_1\cdots x_{m-1} (K[x_{m},\ldots,x_{n-2}]/ (x_m, x_{m+1}\cdots x_{2m}, \ldots, x_{n-m-1}\cdots x_{n-2}) ) 
[x_{n},x_1 \ldots ,x_{m-1}].$$
\normalsize
Indeed, let $u\in J_{n,m}\setminus I_{n,m}$ be a monomial. If $x_{n-m+2}\cdots x_nx_1|u$, then $x_{n-m+1}\nmid u$ and $x_2\cdots x_m\nmid u$. It follows that:
\footnotesize
$$u\in x_{n-m+2}\cdots x_nx_1 (K[x_2,\ldots,x_{n-m}]/(x_2\cdots x_{m},x_3\cdots x_{m+2}\ldots,x_{n-2m+1}\cdots x_{n-m}))[x_{n-m+2},\ldots ,x_n,x_1].$$
\normalsize
If $x_{n-m+2}\cdots x_nx_1\nmid u$ and $x_{n-m+3}\cdots x_nx_1x_2|u$ then $x_{n-m+2}\nmid u$ and $x_3\cdots x_m\nmid u$. Thus:
\scriptsize
$$u\in x_{n-m+3}\cdots x_nx_1x_2 (K[x_3,\ldots,x_{n-m+1}]/(x_3\cdots x_{m},x_4\cdots x_{m+3},\ldots,x_{n-2m+2}\cdots x_{n-m+1}))[x_{n-m+3},\ldots ,x_n,x_1,x_2].$$
\normalsize
Finally, if $x_{n-m+2}\cdots x_nx_1\nmid u$, \ldots, $x_{n-1}x_nx_1\cdots x_{m-2}\nmid u$ and $x_nx_1\cdots x_{m-1}|u$, then it follows that 
$x_{n-1}\nmid u$ and $x_m\nmid u$. Therefore:
\small
$$ u \notin x_nx_1\cdots x_{m-1} (K[x_{m},\ldots,x_{n-2}]/ (x_m, x_{m+1}\cdots x_{2m}, \ldots, x_{n-m-1}\cdots x_{n-2}) ) 
[x_{n},x_1 \ldots ,x_{m-1}].$$ \normalsize
As in the proof of Theorem $3.1$ (see the computations for $V_{k,j}$'s), by applying Lemma $1.1$ and Lemma $1.2$, it follows that $\sdepth(J_{n,m}/I_{n,m})\geq \depth(J_{n,m}/I_{n,m}) = \varphi(n-m-2,m)+m =\psi(n,m)+m-1$, as required.
\end{proof}

Inspired by \cite[Conjecture 1.12]{ciclu} and computer experiments \cite{cocoa}, we propose the following:

\begin{conj}
For any $n\geq 3(m+1)+1$, we have $\sdepth(S/J_{n,m})=\varphi(n,m)$.
\end{conj}

\vspace{2mm} \noindent {\footnotesize
\begin{minipage}[b]{15cm}
Mircea Cimpoea\c s, Simion Stoilow Institute of Mathematics, Research unit 5, P.O.Box 1-764,\\
Bucharest 014700, Romania, E-mail: mircea.cimpoeas@imar.ro
\end{minipage}}

\end{document}